\documentclass[12pt,reqno]{amsart}

\usepackage{color}
\usepackage{amsmath, amsthm, amssymb}
\usepackage{amsfonts}
\usepackage[ansinew]{inputenc}
\usepackage[dvips]{epsfig}
\usepackage{graphicx}
\usepackage[english]{babel}
\usepackage{hyperref}

\usepackage{cite}
\usepackage{graphicx}
\usepackage{amscd}
\usepackage{color}
\usepackage{bm}
\usepackage{enumerate}

\usepackage{verbatim}
\usepackage{hyperref}
\usepackage{amstext}
\usepackage{latexsym}
\theoremstyle{plain}
\newtheorem{Theorem}{Theorem}[section]

\newtheorem{Proposition}[Theorem]{Proposition}
\newtheorem{Lemma}[Theorem]{Lemma}
\newtheorem{Remark}[Theorem]{Remark}

\numberwithin{Theorem}{section}
\numberwithin{equation}{section}

\def\square{\vbox{
\hrule height .4pt \hbox{\vrule width .4pt height 7pt \kern 7pt
\vrule width .4pt} \hrule height .4pt }}
\def\QED{\hfill {$\square$}\goodbreak \medskip}

\newcommand{\average}{{\mathchoice {\kern1ex\vcenter{\hrule height.4pt
width 6pt depth0pt} \kern-9.7pt} {\kern1ex\vcenter{\hrule
height.4pt width 4.3pt depth0pt} \kern-7pt} {} {} }}

\def\R{\mathbb{R}}

\renewcommand{\a }{\alpha }

\newcommand{\D }{\Delta }

\newcommand{\e }{\varepsilon }

\renewcommand{\l }{\lambda }

\newcommand{\n }{\nabla }
\newcommand{\vp }{\varphi }

\newcommand{\s }{\sigma }

\renewcommand{\t }{t} %\renewcommand{\t }{\tau }

\renewcommand{\o }{\omega }
\renewcommand{\O }{\Omega }

\newcommand{\ov}{\overline}

\newcommand{\be}{\begin{equation}}
\newcommand{\ee}{\end{equation}}

\newcommand{\de}{\partial}

\newcommand{\ti}{\widetilde}

\renewcommand{\textbf}[1]{\begingroup\bfseries\mathversion{bold}#1\endgroup}

\newcommand{\N}{\mathbb{N}}

\newcommand{\Z}{\mathbb{Z}}

\newcommand{\cB}{{\mathcal B}}

\newcommand{\cH}{{\mathcal H}}

\newcommand{\cL}{{\mathcal L}}

\newcommand{\cP}{{\mathcal P}}

\newcommand{\cU}{{\mathcal U}}

\renewcommand{\epsilon}{\varepsilon}

\begin{document}

\title[Overdetermined Neumann boundary with a  ]
{An Overdetermined Neumann  boundary  value  problem with a general driving force}

%\author{Mouhamed Moustapha Fall}
%\address{M. M. F.: African Institute for Mathematical Sciences in Senegal, KM 2, Route de
%Joal, B.P. 14 18. Mbour, Senegal.}
%\\email{mouhamed.m.fall@aims-senegal.org}

%\author{Ignace Aristide Minlend}

\author{Ignace Aristide Minlend}
\address{I.A. M.: Faculty of Economics and Applied Management, University of Douala,  BP 2701, Douala, Littoral Province, Cameroon.}
\email{\small{ignace.a.minlend@aims-senegal.org} }

\author{Jing Wu}
\address{J. W.: Departamento de Matem\'aticas,  Universidad Aut\'onoma de Madrid, Ciudad Universitaria de Cantoblanco, 28049 Madrid, Spain.}
\email{jing.wu@uam.es;jingwulx@126.com}

%\\author{Tobias Weth}
%\\address{T.W.:  Goethe-Universit\"{a}t Frankfurt, Institut f\"{u}r Mathematik.
%\Robert-Mayer-Str. 10 D-60629 Frankfurt, Germany.}

%\\email{weth@math.uni-frankfurt.de}

%\\keywords{Schiffer conjecture, Overdetermined problems, Neumann eigenvalue, bifurcation }

\begin{abstract}
In this paper, we prove the existence of  a family of non trivial compact subdomains $\O$ in the manifold $\mathcal{M}=\R^N\times \R/2\pi\Z, N\geq 2$ for which the overdetermined Neumann boundary value problem
 \begin{align}\label{Neumann1}
  \left \{
    \begin{aligned}
       -\D w&=\mu g(w)  && \qquad \text{in $ \Omega$,}\\
      \frac{\partial  w}{\partial\eta}  &=0 &&\qquad \text{on $\partial  \Omega,$}\\
             w&=c\ne 0 &&\qquad \text{on $\partial \Omega$,}
    \end{aligned}
       \right.
\end{align}
 admits solutions for some  $\mu > 0$ and  a  $C^{1, \alpha}$ function $g:\R\rightarrow\R.$ The domains we construct have nonconstant principal curvature, and  therefore are not isoparametric nor homogeneous. The argument we develop applies   for  both  linear  and  non-linear functions  $g$. By this, we generalise  a recent result obtained by Fall, Weth and the first named author in \cite{Fall-MinlendI-Weth4}, where the overdetermined  Neumann eigenvalue problem for the Laplacian was considered.

\end{abstract}
\maketitle

\maketitle

\section{Introduction and main result}
In this paper, we  consider the manifold $\mathcal{M}:=\R^N\times \R/2\pi\Z$  endowed with the flat metric and we  are concerned with the existence of subdomains $\O \subset  \mathcal{M}$   admitting solutions  to the overdetermined Neumann boundary value problem
\begin{equation}\label{h-Neu-ove}
  \left \{
    \begin{aligned}
       -\D w&=\mu g(w) && \qquad \text{in $ \Omega$,}\\
        \frac{\partial  w}{\partial\eta}   &=0 &&\qquad \text{on $\partial  \Omega,$}\\
             w&=c \ne 0&&\qquad \text{on $\partial \Omega$,}
    \end{aligned}
       \right.
\end{equation}
where  $\eta$ is the unit outer  normal to the boundary $\partial \O$, $\mu$ is some positive constant  and    $g:\R\rightarrow\R$ is a $C^{1, \alpha}$ function.

When the manifold $\R^N\times \R/2\pi\Z$ is replaced  by the Euclidean space $\R^N$, and $g$ is a nonlinearity,  problem  \eqref{h-Neu-ove} can be viewed as a  non-linear   counterpart of  a long standing open conjecture by Schiffer \cite[Problem 80, p. 688]{Yau},  which  states that  balls are  the only smooth bounded subdomains $\O \subset \R^N$  such that there exist a constant   $\mu >0$ and a solution  $u\ne 0$  to  the overdetermined Neumann  problem
\begin{align*}%\label{h-Neu-Neumanunbo}
(\textrm{N}_{\mu}):
  \left \{
    \begin{aligned}
       - \D u  &=\mu  u && \qquad \text{in $\Omega$,}\\
           \frac{\partial  u}{\partial\eta}  &=0 &&\qquad \text{on $\partial \Omega$,}\\
              u &= 1  &&\qquad \text{on $\partial \Omega$.}
    \end{aligned}
       \right.
\end{align*}
Schiffer's conjecture  finds  a  strong connection   with the so called   \emph{Pompeiu  problem} \cite{L.BrownSchreiberTaylor, Berenstein, Williams, Zalcman}.  A bounded domain $\Omega \subset \R^{N}$ is said to have  the \emph{Pompeiu property}  if $f \equiv 0$ is the only continuous function on $\R^{N}$ for which
\begin{equation*}%\label{eq:failPompp}
\int_{\sigma(\Omega)}f \,dx = 0 \qquad \text{for every rigid motion $\sigma$ of $\R^N$.}
\end{equation*}
The Pompeiu  problem  consists  in finding all sets $\Omega \subset\R^{N}$  having   the Pompeiu  property.
In  1976,  Williams  \cite[Theorem 2, p. 186]{Williams}, (see also \cite{L.BrownSchreiberTaylor}) proved that a domain $\Omega \subset \R^N$  \emph{ homeomorphic to the unit ball} in $\R^N$ fails  the \emph{Pompeiu property} if and only if  problem  $(\textrm{N}_{\mu})$ admits a  nontrivial solution  for some  $\mu>0$.  In  \cite{G.Liu2},  Liu  also  obtained  a  new necessary and sufficient condition   for a domain without the Pompeiu property. Namely  a bounded and  $C^{4, \alpha}$ domain  $\Omega$ with $\R^N \setminus \overline{\Omega}$ connected fails the Pompeiu property if and only if  there  is a nontrivial \emph{buckling} eigenfunction  with  second order interior normal derivative constant on the boundary.

So far,  a lot of works have been  devoted  to address the validity of Schiffer's conjecture but  positive results  only  exist in   some particular cases. In  \cite{Berenstein, BerensteinYang}, Berenstein and Yang  proved that the existence of a solution to $(\textrm{N}_\mu)$  with  $\mu=\lambda_2$ (the second  Dirichlet eigenvalue of the  Laplacian)  implies  that $\O$ is a ball.   Moreover, the existence of infinitely many eigenvalues to $(\textrm{N}_\mu)$ implies the validity of the conjecture.  In the same direction,  Aviles \cite{Aviles} proved in 1986 that disks are the only  planar and convex domains admitting solutions to the problem  $(\textrm{N}_{\mu})$ for any eigenvalue  $\mu$ less  than the seventh Neumann eigenvalue of the problem. In \cite{G.Liu1}, Liu proved that the Schiffer conjecture is valid for domains   where additionally, the  third order interior normal derivative  of the  corresponding Neumann eigenfunction   is constant  on the boundary.  This  result was extended to overdetermined  problems  with fully nonlinear operator in  \cite{G.Liu2}.  For further  positive results  related to problem  $(\textrm{N}_\mu)$,  we  also refer the reader to \cite{Deng, WillmsGladwell,Friedlander}. Although Schiffer's conjecture is valid in some specific cases, negative results exist as well. For instance,  when one considers the problem  \eqref{h-Neu-ove}  on periodic domains, Fall, Tobias and Minlend \cite{Fall-MinlendI-Weth4} obtained  bifurcations of  straight  cylinders in $\R^N\times \R$ admitting solutions to  $(\textrm{N}_\mu)$. Also in  the very recent  work  \cite{EncisoDavidSicFer},  Enciso, Fernandez, Ruiz and Sicbaldi considered a weaker analog of Schiffer's conjecture which claims  that among domains with  disconnected boundary, balls and annulus  are the only smooth bounded  domains $\O \subset \R^2$  admitting solutions to  $(\textrm{N}_{\mu})$  which are  locally constant  on $\partial \O$. They obtained a negative answer to this conjecture by constructing a family of doubly connected domains $\O$ with the above property.  \\

An interesting question in the literature of overdetermined Neumann problems  is  to know what happens  when one replaces  the  linear   function $w\mapsto \mu w$  in $(\textrm{N}_{\mu})$ with  a  non linearity  $g$.  This question remains less studied and  to  our level of information,    we could  only find  a result  by  Kawohl and Lucia  in  \cite{kawohl},  where  they  proved, provided   $g(c)\ne 0$,  that the boundary  of a solution domain $\O$  to  problem  \eqref{h-Neu-ove}  in $\R^2$  is  a circle if and only if the problem admits a solution  having  a constant third or fourth normal derivative along the boundary $\partial \O$.

We emphasize that  both problem  $(\textrm{N}_{\mu})$ and its  non-linear version  can  be considered in  general Riemannian manifold. In \cite{V.Shklover}, Shklover  proposed a generalization of Schiffer's conjecture to an arbitrary Riemannian manifold with the possibility of replacing the condition  on the domain to be a ball by the  more general  assumption that the domain has homogeneous boundary (i.e  boundary admitting transitive group of isometries).  He then disproved this conjecture in manifolds $\mathcal{M}$ of constant sectional curvature by providing examples of solution  domains  to  $(\textrm{N}_{\mu})$   whose boundaries are isoparametric (see e.g. \cite[Definition 3]{V.Shklover} ) but not homogenous. It is clear by  Cartan's theorem \cite{cartan} that the isoparametricity property of these hypersurfaces implies that their principal curvatures are constant. It was then  tempting to  guess that the boundary  of a solution domain   to  $(\textrm{N}_{\mu})$  in any Riemennian manifold is  an isoparametric hypersurface. In the recent paper \cite{Fall-MinlendI-Weth4}, Fall, Weth and the first named author showed that this is not the case  by  providing counterexamples in  the manifold  $\mathcal{M}:=\R^N\times \R/2\pi\Z$. In this paper, we  use bifurcation theory   and   extend this result  to  the problem  \eqref{h-Neu-ove}. We underline that solving  overdetermined  Neumann boundary problem  via bifurcation approach   comes with a loss of derivatives. A  strategy to overcome this drawback   was  developed  in \cite{Fall-MinlendI-Weth4}  at least in the context of overdetermined problems and was  successfully applied  in  \cite{EncisoDavidSicFer}. %The argument we use in this work has  the potential to apply  for  general nonlinearities  $g$ satisfying  $g(0)=0$ and for which the Dirichlet problem
%\begin{align*}
 % \left \{
  %  \begin{aligned}
   %    -\D w&=g(w) && \qquad \text{in $ \Omega$,}\\
    %         w&=c \ne 0&&\qquad \text{on $\partial \Omega$,}
    %\end{aligned}
     %  \right.
%\end{align*}
%admits a radial solution  with at least one critical point. We postpone this investigation to a future project.

The aim of this paper is to carry out such a construction under relatively minimal assumptions on the  function $w\mapsto g(w)$. A necessary requirement is the existence of a solution to the overdetermined Neumann in the unit ball. Hence, we make the following assumption:\\

\textbf{Assumption (A):}
\begin{enumerate}
\item
There exists a  radial  solution $U_*\in C^{2, \alpha}(\overline{B_1})$  to the overdetermined  problem
\begin{equation}\label{h-Neu-ove33}
  \left \{
    \begin{aligned}
       -\D U&=\mu_* g(U) && \qquad \text{in $ B_1$,}\\
        \frac{\partial  U}{\partial\eta}   &=0 &&\qquad \text{on $\partial  B_1,$}\\
             U&=c_* \ne 0&&\qquad \text{on $\partial B_1$,}
    \end{aligned}
       \right.
\end{equation}
for some $\mu_*>0$.
 \item
The  radial solution  $U_*(|t|)=U_*(r)$ is  well defined on $[0, +\infty)$.
\item
The composition $g'\circ U_*: \R \to \R$ is $C^{1, \alpha}$.
\end{enumerate}

\begin{Remark}\label{eqRK1}
\begin{enumerate}
\item
Note that Assumption (A)-(iii) is meant to ensure that the linear operator in Proposition \ref{sec:functional-setting-4} is well defined. Actually (A)-(iii) is fulfilled  for any $C^2$ function $g$.
\item
Assumption (A)-(i) and (A)-(ii)  are fulfilled   by  any  function  $g$ for which the  equation
\begin{equation}\label{Diric}
-\D u=g(u) \quad \text{in $\R^N$}
\end{equation}
admits a  radial solution  $u(t)= u(|t|)$ which has  a  critical  point $\rho_*\in (0, +\infty)$, with  $u(\rho_*)\ne 0$.
Indeed setting \begin{equation*}%\label{eq.radialsolu144}
 U_{*}(r):=u(\rho_{*}r), \quad r\geq 0,
\end{equation*}
 we see that  $U_{*}$ solves  \eqref{h-Neu-ove33} with
\begin{equation*}%\label{eqcosnta44}
\mu_*:=\rho_*^2 \quad \textrm{ and }\quad c_*:= U_{*}(1).
\end{equation*}
\end{enumerate}
\end{Remark}

To state our main result, we fix $\alpha \in (0,1)$ and define  by  $C^{2,\alpha}_{p}(\R)$  the space of  $2\pi$ periodic and even $C^{2,\alpha}$-functions on $\R$, and we let $\cP^{2,\alpha}_{p}(\R)$ denote the open subset of strictly positive functions in $C^{2,\alpha}_{p}(\R)$. For a function $h \in \cP^{2,\alpha}_{p}(\R)$, we  define  the domain
\begin{equation}\label{eq:PertTorus}
\Omega_h:= \left\{\left(t,x \right)\in  \R^{N}\times \R \::\: |t|<\frac{1}{h(x)}  \right\}.
\end{equation}
Our main result reads

\begin{Theorem}\label{Theo1-ND}
Let $N\geq2$ and $g: \R\rightarrow\R$ is a $C^{1}$  function such that Assumptions (A) holds. Then  there exist  some constants  $\mu_*,\e_*,  \l_* >0$, $ c_*,\beta_*,\delta_* \in \R \setminus \{0\}$, depending only on $N$, and a smooth curve
$$
(-{\e_*},{\e_*}) \to   (0,+\infty) \times  \cP^{2,\alpha}_{p}(\R) ,\qquad s \mapsto (\l_*(s),h^*_s)
$$
with $\l_*(s) \big|_{s=0}= \lambda_*$,
$$
h^*_s(x)= \frac{1}{\sqrt{\l_*(s)}} + s \beta_* \cos(x) + o(s) \qquad \text{as $s \to 0$ uniformly on $\R$,}
$$
and the property that the overdetermined boundary value problem
\begin{equation}\label{eq:solved-main}
  \left \{
    \begin{aligned}
       -\D w_s&= \frac{\mu_*}{\l_*(s)}  g(w_s)  && \qquad \text{in $ \Omega_{ h^*_s}$,}\\
                \frac{\partial  w_s}{\partial\eta}  &=0 &&\qquad \text{on $\partial  \Omega_{ h^*_s}$,}\\
                w_s&=c_* &&\qquad \text{on $\partial \Omega_{ h^*_s}$,}
    \end{aligned}
       \right.
\end{equation}
admits a classical solution $w_s$ for every $s \in (-\e_*,\e_*)$, which is radial in $t$, even and  ${2\pi} $-periodic in $x$.  Moreover, we have
\begin{equation} \label{eq: expansion  of the solu}
%\label{eq:w-s-expansion1}
w_s\left(\frac{t}{h^*_s(x)},x\right)= U_*(|t|) +s\bigl\{U_*(|t|)+ \delta_* \,|t| U_*'(|t|)\bigr\} \cos(x) + o(s)  \quad \text{as $s \to 0$}
\end{equation}
uniformly on $B_1 \times \R$, where $t \mapsto U_*(|t|)$ is a suitable radial function given in Assumptions (A).
\end{Theorem}

\begin{Remark}\label{eqRK0}
We note that the parameter $\lambda_*$ is the first Dirichlet eigenvalue of the operator $\mathcal{K}_*:  C^{2,\a}_{0,rad}(B_1)  \to C^{0,\a}_{rad}(B_1)$ defined by
\begin{equation}\label{eqcK*}
\mathcal{K}_*(u):=\Delta_{t}u+\mu_{*}g'(U_*)u,
\end{equation}
where $C^{2,\a}_{0,rad}(B_1)$  denotes  the space of radial  functions in $C^{2,\a}(B_1)$  with  vanishing values  at the boundary of $B_1$, and corresponding eigenfunction $V_*$. Moreover,
\begin{align*}
\delta_*& = -\frac{V'_*(1)}{U''_*(1)}\quad \textrm{and}\quad \beta_* = -\frac{\delta_*}{\sqrt{\lambda_*}}.\\
\end{align*}
\end{Remark}
Examples of  a  function $g$  fulfilling  the Assumption (A)  are:

\begin{enumerate}
\item[1)] $g(u)=|u|^{p-1}u, \, \textrm{for some}\quad  p>1.$
 To see this, we note that from \cite[ problem (IVP) Section 2]{nagasaki} (see also \cite[Proof of Theorem 1, pp. 222-223 ]{nagasaki}), we have that  when
\begin{align}\label{value N1}
\textrm{$N \geq 3$ and $p\in \left(1, \frac{N+2}{N-2}\right)$ or ($N=2$  and $p>1$)},
\end{align}
for each  $ n \in \N$, there exists a positive number $A_{n }$ for which the ODE
 \begin{equation}\label{eqODE}
  \left \{
    \begin{aligned}
  &(r^{N-1}u')'+r^{N-1}|u|^{p-1}u=0&& \qquad \text{in $(0,1)$,}\\
  &u(0)=A_{n },\,\,u'(0) =0,
    \end{aligned}
       \right.
     \end{equation}
admits a unique solution $u=u_n$  which has $(n-1)$ zeroes $0<r_1< \cdots< r_{n-1}$ in the interval $(0, 1)$ and vanishes at $r=r_n:=1$.  Hence, there exists  a strictly increasing sequence of real numbers  $(\mu_{n, m})_{1\leq m\leq  n-1},$  with   $0<r_m< \mu_m  <r_{m+1}$ for $m=1, \cdots, n-1$, such that $u'(\mu_{n, m})=0$ and $u(\mu_{n, m})\neq 0.$ Therefore, following   the Remark \eqref{eqRK1},  for each $m=1, \cdots, n-1$,  we have that the radial  function
 $U_{n, m}$  defined by
 \begin{equation}\label{eq:e22}
 U_{n, m}(r):=u(\mu_{n, m}r),
 \end{equation}
solves \eqref {h-Neu-ove33}
with
$$
\mu_*:=\mu_{n, m}^2 \quad \textrm{and }\quad c_*:=C_{n, m}= U_{n, m}(1).
$$
Furthermore, we  note that  the solution $ U_{m}$  is well  defined for all $r\in (0, +\infty)$. Indeed, the  existence  of the  solution $u=u_n$  to the ODE \eqref{eqODE} (which is involved in \eqref{eq:e22})  is  related to the existence of the function $w$ defined in  \cite[ from (2.1) and (2.5)]{nagasaki}  by $$w(\rho)=A^{-1}_n \rho^{\frac{2}{p-1}} u(r), \textrm{with } r=B_n e^\rho,$$
 where $B_n=\{(\frac{2}{p-1}+2-n)\frac{2}{p-1}A_n^{1-p}\}^{\frac{1}{2}}.$ The function $w$ is well defined on $(-\infty, +\infty)$ and furthermore, \cite[Proof of Proposition 2.2, pp. 216-217]{nagasaki}, $w$ tends  to zero at $\rho=-\infty$ and $w(\rho)>0$ for  $\rho$ large enough. One can also refer to \cite{nagasaki88}.
Notice that Assumption (iii) holds  provided  $p\geq 2$ since this implies $g$ is $C^{2}.$ But then we are restricted to dimensions $N\in \{2, 3, 4, 5, 6\}$ in \eqref{value N1}. 

\item[2)]
$g(u)=-u+ u^{p}, \, \textrm{for some} \,  p>1.$
From \cite[Theorem 1.6]{Ni}, if $p\geq \frac{N+2}{N-2}$, then for each $\beta  >0$ and $\beta \ne 1$,
the  ODE
 \begin{equation}\label{eqODE1}
  \left \{
    \begin{aligned}
  & u''+ \frac{N-1}{r} u'+g(u)=0&& \qquad \text{in $[0,+\infty)$,}\\
  &u(0)=\beta >0,\,\,u'(0) =0,
    \end{aligned}
       \right.
     \end{equation}
admits a  solution $u=u_{\beta}$ which stays positive for all $r\in (0, +\infty)$ and has infinitely many local maxima and local minima at points $(\gamma_{\beta ,n})_{n\in \N}$. Arguing as in \eqref{eq:e22}, we have a solution
 \[U^{\beta }_{n}(r):=u(\gamma_{\beta , n}r)\]
of \eqref{h-Neu-ove33} with
$$
\mu_*:=\gamma_{\beta , n}^2 \quad \textrm{ and }\quad c_*:=u(\gamma_{\beta , n})> 0.
$$
Note  also that   since the solutions of  \eqref{eqODE1} are  positive,  $g$ is $C^2$ for  $ p>1$ and  the assumption (A)-(iii) holds.

\item[3)]$g(u)=u-u^{3}.$ By \cite[Theorem 1 and Example 1]{Kajikiya}, there exist $a, b>0$ such that for  all $\beta \in (-b, a) \setminus \{0\}$,   the problem
 \begin{equation*}%\label{eqODE5}
  \left \{
    \begin{aligned}
  & u''+ \frac{N-1}{r} u'+g(u)=0&& \qquad \text{in $[0,R)$,}\\
  &u(0)= \beta \neq0,\,\,u'(0)=u(R)=0,
    \end{aligned}
       \right.
     \end{equation*}
admits a  unique global radial solution solution $u=u(r, \beta)$  defined on  $[0,\infty)$ and has an infinte  increasing sequence of critical points  $\{\nu_{\beta,k}\}_{k\in \N}$. We choose the $\ell\in \N$ for which $u(\nu_{\beta,\ell})\ne 0$ and argue as  in \eqref{eq:e22} to get a solution $U^{\beta }_{\ell}(r)$ to \eqref{h-Neu-ove33}.

\item[4)] As $g$ satisfies the following conditions:
\[sg(s)>0,\qquad s(g(s)s)'<0  \qquad \mbox{as}\,\, s\neq 0,\]
then for each  $ n \in \N$, the equation
\begin{equation*}%\label{eqODE3}
  \left \{
    \begin{aligned}
  &(r^{N-1}u')'+r^{N-1}(g(u))=0&& \qquad \text{in $(0,1)$,}\\
  &u(0)>0,\,\,u'(0) =u(1)=0,
    \end{aligned}
       \right.
     \end{equation*}
admits a unique solution $u=u_n,$  which has $(n-1)$ zeroes  in the interval $(0, 1),$ see \cite{S.Tanaka09}. A $C^2$ function $g$ satisfying the above mention conditions is given  by $g(s)=\arctan(s)$, see  \cite[p. 5258]{S.Tanaka09}. Moreover  from \cite[Proof of Theorem 1.1, p. 5258]{S.Tanaka09}, the solution $u=u_n$ exists on $[0, +\infty).$
\item[5)]
Assumption (A) includes linear function  $g(u)=\gamma u$ as well. Indeed  problem \eqref{h-Neu-ove33} is solved  in this case by  the functions
\begin{equation*}
  %\label{eq:def-tilde-un}
r \mapsto U_n(r)=\dfrac{I_{N/2-1}(j_nr)}{I_{N/2-1}(j_n)}
\end{equation*}
with
\begin{equation*}
  %\label{eq:bessel-zeros}
\mu_*= \frac{j_{n}^2 }{\gamma} \qquad \text{for some $n \geq 1$,}\quad \textrm{and}\quad  c_*=1.
\end{equation*}
Here, $$
I_\nu(r):= r^{-\nu} J_\nu(r) \qquad \text{for $\nu>-1, r>0$,}
$$ $J_{\nu}$ denotes the Bessel function of the first kind of order $\nu>-1$, and
$$
0< j_{\nu,1} < j_{\nu,2} < j_{\nu,3} < \dots
$$
denote the ordered sequence of zeros of $J_\nu$ and we have set
$$
j_n:= j_{\frac{N}{2},n} \qquad \text{for $n \geq 1$.}
$$
Note in particular that
$$
U_n'(1)=   \dfrac{j_n I_{N/2-1}'(j_n)}{I_{N/2-1}(j_n)}=  -  \frac{j_n^2 I_{N/2}(j_n)}{I_{N/2-1}(j_n)}=0 \quad \text{for $n\geq 1$}
$$
since
 \begin{equation*}%\label{eq:derrec}
  I_\nu'(r) =-r I_{\nu+1}(r) \qquad \text{for $\nu>-1, r>0$.}
 \end{equation*}
 Observe that for $\gamma=1$, we recover the result  by Fall, Weth and the first named author in \cite{Fall-MinlendI-Weth4}.
\end{enumerate}

It is obvious that solutions to  \eqref{eq:solved-main} change sign provided $U_*$ changes sign. It is important  to note that the study of  overdetermined  is often related to  the existence of sign-changing solutions,  but only  few results are known  for overdetermined  boundary value problems in unbounded domains. A part from  the current work, we can only cite the references   \cite{Fall-MinlendI-Weth4, DaiandY.Zhang, M22},  where  different  families of sign-changing solutions were obtained in the context of  overdetermined   problems and in unbounded domains.  Interested reader may find existence results for sign-changing solutions in bounded  domains  in  \cite{ Deng, BCanutoDRial, B.Canuto, Ruiz-arxiv}. \\

Related to this paper are Dirichlet counterparts of the problem  \eqref{eq:solved-main}. A substantial body of research addresses this problem. We mention some of them here and refer readers to the references therein for further exploration of related results. In \cite{dPPW15}, the authors studied problem
\begin{equation}\label{h-D-over3}
  \left \{
    \begin{aligned}
       -\D w&=f(w) && \qquad \text{in $ \Omega$,}\\
             w&=0 &&\qquad \text{on $\partial \Omega$,}\\
    \frac{\partial  w}{\partial\eta}  &=c\ne 0&&\qquad \text{on $\partial  \Omega,$}
    \end{aligned}
       \right.
\end{equation}
for the   Allen-Cahn nonlinearity $f(u)=u-u^3$, but in domains that are perturbations of a dilated straight cylinder, i.e. perturbations of $(\epsilon^{-1}\,B_1)\times\mathbb{R}$ for $\epsilon$ small, or more in general domains that are perturbations of a dilation of the region contained in an onduloid. Recently, Ruiz, Sicbaldi and the named second author proved that there exist nontrivial unbounded domains, bifurcating from the straight cylinder, where the overdetermined elliptic problem  \eqref{h-D-over3}  admits a positive bounded solution for a very general class of functions $f$, see \cite{RSW21}. Another type of construction has been given in \cite{RRS20}, where Ros, Ruiz and Sicbaldi show that  \eqref{h-D-over3} admits a solution for some nonradial exterior domains for $f(u)=u^p-u$, $1<p<\frac{N+2}{N-2}$.\\ %One can refer to the references therein for further results.\\

We now explain the proof of  Theorem \ref{Theo1-ND}  while presenting the organization of the paper. Theorem \ref{Theo1-ND}  is proved   applying  the  Crandall-Rabinowitz bifurcation theorem, \cite{M.CR}. Following  Section  \ref{sectio1}, we have  to  solve  the problem  \eqref{h-Neu-over}  on  domains of the form  $\O_h$  defined by  \eqref{eq:PertTorus}. This is equivalent  to solving the $\lambda$-dependent problem   \eqref{eq:perturbed-strip-ND-0} which we rephrase to the  problem   \eqref{eq:Proe1-ss2-ND} on the fixed domain $\Omega_*=B_1\times \R$ with a  second order nonlinear   operator $L^h_\lambda$ given by  \eqref{eq:reldiffope-ND-alt}. Under the functional setting in Section  \ref{sec:functional-setting},  we are  led to considering  the functional equation  $F_\lambda(u,h)=L^{1+h}_\lambda(u_*+u)=0$ with unknown functions $u \in C^{2,\alpha}_{p,rad}(\overline{\O_*})$ and $h \in C^{2,\alpha}_p(\R)$ for some $\alpha \in (0,1)$, where
$C^{2,\alpha}_{p,rad}(\overline{\O_*})$ denotes the space of $C^{2,\alpha}$-functions $u=u(t,x)$ which are radial in $t$ and $2\pi$ periodic and even in $x$ and   $u_*$ is defined by \eqref{eq.radialsolu2}.  As already explained in  \cite{Fall-MinlendI-Weth4}, the overdetermined  Neumann boundary problem  comes with a loss of derivatives  which prevents  the linearization of $F_\lambda$ at $(0,0)$  to be  of  Fredholm type when defined  between classical H\"older spaces. To  bypass this challenge, we had  to use assumption (A)-(ii) to produce  a more  accurate solution form to the equation   $F_\lambda(u,h)=0$,  which allowed us   to  express the unknown  $h$ as a function of $u$,  see Remark  \ref{rem:const-sol}.  By substituting $h=h_u$ in $F_\lambda(u,h)$, we  reduce our problem to an equation of the type $G_\lambda(u)=0$ for some function $(\lambda,u) \mapsto G_\lambda(u)$, see \eqref{eq:DeffGl-ND}.  Note that the unknown   $h=h_u$ in this involving first order derivative of $u$, see \eqref{eq:DeffM-ND--2}.  Since we need  $h \in C^{2,\alpha}_p(\R)$,  we therefore  have to consider both $F$ and $G_\lambda$ as maps between (open subsets)  of new tailor made Banach spaces $X_2^D$ and $Y$, see Section~\ref{sec:functional-setting} below. In Proposition \ref{sec:functional-setting-4}, we  compute the linearised operator  $D_uG_\lambda(0): X_2^D \to Y$ and show   in  Proposition \ref{fredholm} that it is a Fredholm operator of index zero.  We note that, this is precisely where the assumption (A)-(iii)  is used  to ensure   $D_uG_\lambda(0)$ and the operator $\mathcal{M}_*$ in the proof of Lemma 3.5 are well defined. In  Section \ref{section 4},  we show that the first  Dirichlet eigenvalue of the operator  $ \mathcal{K}_{*}$ given by \eqref{eqcK*} yields the bifurcation parameter $\lambda_*$ for which  $D_uG_{\l_{*}}(0):X_2^D\to Y$  has a one  dimensional kernel and the transversality condition  in the Crandall-Rabinowitz bifurcation theorem \cite{M.CR} holds.\\

\noindent \textbf{Acknowledgements}: I.A.M. is supported by the return fellowship of the Alexander von Humboldt Foundation and J.W. is supported by Proyecto de Consolidaci\'{o}n Investigadora 2022, CNS2022-135640, MICINN. Part of this work was carried out when the authors were visiting the Goethe University Frankfurt am Main. They are grateful to the Mathematics department for the hospitality and wish to thank Prof. Tobias Weth for valuable comments throughout the writing of this paper.

\section{The pull back problem}\label{sectio1}
%\subsection{Preliminaries}

Recall that we are looking for a non-constant function $h \in \cP^{2,\alpha}_{p}(\R)$ with the property that the overdetermined problem
\begin{equation}\label{h-Neu-over}
  \left \{
    \begin{aligned}
       -\D w&=\mu g(w)&& \qquad \text{in $ \Omega_{h}$,}\\
             w&=c\ne 0  &&\qquad \text{on $\partial \Omega_{h}$,}\\
               \frac{\partial  w}{\partial\eta}   &=0 &&\qquad \text{on $\partial  \Omega_{h}$},
    \end{aligned}
       \right.
\end{equation}
admits a solution   with  $\mu>0$. For a parameter $\lambda>0$,  we defined the operator
$$
L_{\lambda,\mu}u:=  \D_{\t}u + \lambda \partial_{xx}u +  \mu g(u).
$$
Then  it is  straightforward  to check that  a function $u \in C^2(\Omega_h)$ is a solution of
\begin{equation}
  \label{eq:perturbed-strip-ND-0}
  \left \{
    \begin{aligned}
  L_{\lambda,\mu} u &=  0 && \qquad \text{in $\Omega_h$,}\\
           u&=c\ne 0 &&\qquad \text{on $\partial \Omega_h$,}\\
              \frac{\partial  u}{\partial\eta}   &=0 &&\qquad \text{on $\partial \Omega_h$.}
    \end{aligned}
       \right.
     \end{equation}
if and only if the function
\begin{equation}
  \label{eq:changsoluD}
w^{\lambda}\in C^2(\Omega_h),\qquad   w^{\lambda} (t, x):=u\left(  \frac{t}{\sqrt{\lambda}}, x\right)
\end{equation}
solves (\ref{h-Neu-over}) with $h$ replaced by $\frac{h}{\sqrt{\lambda}}$ and $\mu$ replaced by   $\frac{\mu}{\lambda}$

We  also  observe that in the  special case $h \equiv 1$ in \eqref{eq:perturbed-strip-ND-0},  $\Omega_h$ is    the straight cylinder $\Omega_h = \Omega_1 = B_1 \times \R$ and by  Assumption (A)-(i), there exists a solution  $U_{*} \in C^2([0,1])$  to the  (overdetermined) ODE problem
\begin{equation}
  \label{eq:ODE-eigenvalue}
U''  + \frac{N-1}{r} U'+\mu_* g(U) = 0 \quad \text{in $(0,1)$,}\qquad U'(0)=U'(1)=0,\quad U(1)=c_*\ne 0.
\end{equation}
Consequently,  we   have a solution
\begin{equation}\label{eq.radialsolu2}
 u_{*}(t, x):=U_{*}(|t|)
\end{equation}
of   \eqref{eq:perturbed-strip-ND-0}  in the fixed domain
$$ \Omega_*:= \Omega_1=B_1\times \R.$$
In the following, we  put
 \begin{equation*}%\label{eq:opLlam}
L_{\lambda}u:= L_{\lambda,\mu_*}u= \D_{\t} u+ \lambda \partial_{xx}u +  \mu_* g(u).
 \end{equation*}

Observe that, for a function $h \in \cP^{2,\alpha}_{p}(\R)$, the domain $\Omega_h$ is parameterized by the mapping
 $$ \Psi_h: \Omega_*  \to  \Omega_h , \quad  ( t,x)  \mapsto (\tau, x)=\left( \frac{\t}{h(x)},x\right),$$
with inverse
$$
\Psi^{-1}_h:  \Omega_h   \to  \Omega_*, \quad  ( \tau ,x)  \mapsto   ( h(x) \tau ,x).
$$
Hence \eqref{eq:perturbed-strip-ND-0} is equivalent to
\begin{equation}
  \label{eq:perturbed-strip-ND-fixed-domain}
  \left \{
    \begin{aligned}
  L_{\lambda}^h u &=  0 && \qquad \text{in $\Omega_*$,}\\
           u &=c_* &&\qquad \text{on $\partial \Omega_*$,}\\
             |\n u| &=0 &&\qquad \text{on $\partial \Omega_*$,}
    \end{aligned}
       \right.
     \end{equation}
     where the operator
     \begin{equation}
     \label{eqreladiffopets-ND}
     L_{\lambda}^h: C^{2}(\ov{\Omega_*}) \to C^{0}(\ov{\Omega_*})\quad \text{is defined by}\quad
     L_{\lambda}^h u= \bigl(L_\lambda (u \circ \Psi_h^{-1})\bigr)\circ \Psi_h.
     \end{equation}
     Indeed, $u \in C^{2}(\ov{\Omega_*})$ solves (\ref{eq:perturbed-strip-ND-fixed-domain}) if and only if $u \circ \Psi_h^{-1}$ solves (\ref{eq:perturbed-strip-ND-0}). To calculate an explicit expression for $L_{\lambda}^h$, we fix $u\in C^{2}(\ov{\Omega_*})$ and note that
\begin{align*}%\label{eqreladiffopets-ND-1}
[L^h_\lambda u] (h(x)t,x) = [L_\lambda v_h](t,x) \quad \text{for $(t,x) \in \Omega_h$}
\end{align*}
with
\begin{align}
  \label{eqreladiffopets-ND-2}
v_h \in C^2(\ov{\Omega_h}),\qquad v_h(t,x)=u(h(x)t,x).
\end{align}
A direct computation yields
\begin{align*}
L_\lambda  v_h(\t,x) =&\lambda \partial_{xx}u(h(x)t,x) +   h^2(x)\Delta_{t} u(h(x)t,x)\\
                      &+ \lambda h'(x)^2 \nabla^2_t u(h(x) t,x)[t,t] + 2\lambda h'(x) \nabla_{t}\partial_{x}u(h(x) t,x)\cdot t\\
  &+ \lambda h''(x) \nabla_t u(h(x) t,x)\cdot t +\mu_* g(u)\qquad  \text{for $(\t,x) \in \Omega_h$.}
\end{align*}
Replacing $t$ by $\frac{t}{h(x)}$ therefore gives
\begin{align}\label{eq:reldiffope-ND}
L_\lambda^h u(\t,x)& =\tilde{ L}_\lambda^h u(\t,x)+\mu_* g(u) \qquad  \text{for $(\t,x) \in \Omega_*$,}
\end{align}
where
\begin{align*}%\label{eq:ropeLtilde}
\tilde{ L}_\lambda^h u(\t,x) =&\lambda \partial_{xx}u(t,x) +   h^2(x)\Delta_{t} u(t,x)   \\
                      &+ \lambda\frac{h'(x)^2}{h^{2}}  \nabla^2_t u( t,x)[t,t] + 2\lambda \frac{h'(x)}{h(x)} \nabla_{t}\partial_{x}u( t,x)\cdot t \\
  &+ \lambda\frac{ h''(x)}{h(x)} \nabla_t u(t,x)\cdot t \qquad  \text{for $(\t,x) \in \Omega_*$.}
\end{align*}
Here $\nabla_t$ and $\Delta_t$ denote the gradient and Laplacian with respect to the variable $\t \in \R^N$, and we have set
\begin{equation}\label{eq:def-dt-u}
[{D}_\t v](t,x)= \nabla_\t v(t,x) \cdot \t \qquad \text{for functions $v \in C^1(\ov{\Omega_*})$.}
\end{equation}
We also note that
\begin{equation*}%\label{nabla-delta-comp-1}
[{D}_\t {D}_\t v](t,x)={D}_t v(t,x)
+ \nabla^2_t v(t,x)[\t,\t] \qquad \text{for $v \in C^2(\ov{\Omega_*})$.}
\end{equation*}
Hence  \eqref{eq:reldiffope-ND} reads shortly
\begin{align}
L_\lambda^h u &=  \mu_* g(u) + \lambda \partial_{xx}u +   h^2\D_\t u+ \lambda     \frac{(h')^2}{h^2}{D}_\t {D}_\t u   \nonumber\\
&\qquad + 2 \lambda  \frac{ h'}{h}   {D}_{t} \partial_{x}u +  \lambda  \Bigl(\frac{ h''}{h}-\frac{h'^2}{h^2}\Bigr)  {D}_{t}u \qquad  \text{in $\Omega_*$,} \label{eq:reldiffope-ND-alt}
\end{align}
where  we identify the function $h \in \cP^{2,\alpha}_{p}(\R)$ with the function $(\t,x) \mapsto h(x)$ defined on $\ov{\Omega_*}$, and we do the same with $h'$ and $h''$.\\

Hence (\ref{eq:perturbed-strip-ND-fixed-domain}) is equivalent to
\begin{align}\label{eq:Proe1-ss2-ND}
 \begin{cases}
L^h_\lambda u= 0 & \quad \textrm{ in}\quad  \Omega_*, \\
u=c_*&   \quad \textrm{ on}\quad \partial \Omega_*, \\
{D}_\t u = 0  &  \quad \textrm{ on}\quad \partial \Omega_*,
  \end{cases}
  \end{align}
where $L^h_\l$ is given by \eqref{eq:reldiffope-ND-alt}.

\section{Functional setting}
\label{sec:functional-setting}
In this section, we introduce  the spaces where  problem \eqref{eq:Proe1-ss2-ND} will be solved. We also derive important results related to the linearised operator of  $L_\lambda^h$ in  \eqref{eq:reldiffope-ND-alt}.
For fixed $\a\in (0,1)$ and $k \in \N \cup \{0\}$, we set
%$$
%C^{k,\a}_{rad}(\overline \Omega_*):= \{ u \in C^{k,\alpha}(\overline \Omega_*)\::\: \text{$u$ is radial in ${\t}$} \},
%$$
$$
C^{k,\a}_{p,rad}(\overline \Omega_*):= \{ u \in C^{k,\alpha}(\overline \Omega_*)\::\: \text{$u$ is radial in ${\t}$, $2\pi$ periodic and  even in $x$ } \},
$$
endowed with the norm $
u \mapsto \|u\|_{C^{k,\alpha}}:= \|u\|_{C^{k,\alpha}(\ov{\O_*})}.
$ Next, we  define
$$
X_k:= \{ u \in C^{k,\a}_{p,rad}(\overline \Omega_*)\::\:  {D}_\t u \in C^{k,\alpha}(\overline \Omega_*) \},
$$
endowed with the norm
$$
u \mapsto \|u\|_{k}:= \|u\|_{C^{k,\alpha}}+\|{D}_\t u \|_{C^{k,\alpha}}.
$$

\begin{Remark}
  \label{sec:functional-setting-2}
  In the case $k=0$, the existence of the directional derivative ${D}_\t u$,
    defined e.g. by
    $$
    {D}_\t u(t,x)= \frac{d}{d\s}\Bigl|_{\s=1} u(\sigma t,x) \qquad \text{for $(t,x) \in \Omega_*$,}
    $$
    and its $C^\alpha$-continuity up to the boundary is assumed by definition for $u \in X_0$.
\end{Remark}\vspace{0.5ex}
Next, we  also  consider the closed subspaces
$$
X_k^D:=\{u \in X_k \::\: \text{$u= 0$ on $\partial \Omega_*$}\},
$$
and
$$
X_k^{DN}:=\{u \in X_k \::\: \text{$u={D}_\t u = 0$ on $\partial \Omega_*$}\},
$$
both also endowed with the norm $\|\cdot\|_k$ and define the space
$$
Y:= C^{1,\alpha}_{p,rad}(\overline \Omega_*) + X_0^D \; \subset \; C^{0,\alpha}_{p,rad}(\overline \Omega_*),
$$
which is endowed with the norm
$$
\|f\|_{Y}:= \inf \Bigl \{\|f_1\|_{C^{1,\alpha}} + \|f_2\|_{0} \::\: f_1 \in C^{1,\alpha}_{p,rad}(\overline \Omega_*), \; f_2 \in X_0^D,\; f= f_1 + f_2 \Bigr\}.
$$
Under this setting, we  consider the open set
\begin{equation*}%\label{eq:def-cU}
\cU_0:=\{h \in C^{2,\alpha}_{p}(\R) \::\: h>-1\}
\end{equation*}
and define the operator
\begin{align}\label{eq:MappinmgF-ND}
F_\lambda:   X_2^{DN} \times \cU_0 \to Y, \qquad F_\lambda(u, h)= L_\lambda^{1+ h} (u+u_*),
  \end{align}
where $u_*(t,x) = U_*(|t|)$ is  given by \eqref{eq.radialsolu2}.
From (\ref{eq:reldiffope-ND-alt}), we can write  $F_\lambda = F_\lambda^1 + F_\lambda^2$,
where
\begin{align*}
  &F_\lambda^1(u,h)=  \mu_* g(u+u_*) +  (1+h)^2|t|^2 \D_t(u+u_*)   \nonumber\\
  &\qquad \quad \;\;\,+ \lambda     \frac{(h')^2}{(1+h)^2}{D}_{t} D_{t}(u+u_*) + 2 \lambda  \frac{ h' }{1+h}   {D}_t \partial_{ x  }u,\nonumber\\
&F_\lambda^2(u,h) = \lambda \partial_{xx}u+  (1+h)^2(1-|t|^2) \D_t(u+u_*) + \lambda  \Bigl(\frac{ h''}{1+h}-\frac{h'^2}{(1+h)^2}\Bigr){D}_t(u+u_*).
\end{align*}
With this splitting, we obtain  as in   \cite[Lemma 3.3]{Fall-MinlendI-Weth4} the following lemma.
\begin{Lemma}
  \label{sec:functional-setting-1}  The map
  $$
  (u,h) \mapsto F_\lambda(u, h)= L_\lambda^{1+ h} (u+u_*)
  $$
  maps $X_2^{DN} \times \cU_0$ into $Y$.
\end{Lemma}

We  observe   with \eqref{eq:Proe1-ss2-ND} and \eqref{eq:MappinmgF-ND} that,
{\em if $F_\lambda(u, h)=0$, then the function $\tilde u = u_* + u$ solves the problem
\begin{equation}
  \label{eq:equivalence-F-ND}
  \left \{
    \begin{aligned}
         L_\lambda^{1+  h} \ti u & =  0 && \qquad \text{in $\Omega_*$,}\\
           \ti u&= c_* &&\qquad \text{on $\partial \Omega_*$,}\\
         D_\t \ti u&= 0 &&\qquad \text{on $\partial \Omega_*$.}
    \end{aligned}
       \right.
\end{equation}}
We  wish to  further reduce problem   \eqref{eq:equivalence-F-ND}  to the search of a single unknown variable $u$.  This is achieved from the following remark which provides  the solution form to \eqref{eq:equivalence-F-ND} by eliminating the parameter $h$.

\begin{Remark}\label{rem:const-sol}
By Assumption (A)-(ii), we  can extend the function  $u_*$ in \eqref{eq.radialsolu2} to all of $\R^{N} \times \R$ keeping the same  notation   for the extension   and  setting $u_*(t,x):=U_*(|t|)$ for all  $(t, x)\in \R^{N} \times \R$. Then we have $L_\l u_*=0$ in $\R^{N} \times \R$ and therefore, for fixed $h \in C^{2,\alpha}_{p}(\R)$, it follows  from \eqref{eqreladiffopets-ND} that
  \begin{equation}
\label{eq:perturbed-u-n-h}
L^{1+ h}_\l (u_*^h)= 0 \qquad \text{with $u_*^h \in C^{2,\alpha}(\ov{\Omega_*}),\quad$ $u_*^h(t,x)= u_*\left(\frac{t}{1+h(x)}, x\right)$}.
\end{equation}
 Moreover,
\begin{equation}
\label{eq:approximation-u-n-h}
u_*^h = u_* - w_h + O(\|h\|_{C^{2,\alpha}}^2),
\end{equation}
where
\begin{equation*}%\label{eq:approximation-u-n-h-1}
w_h(t,x):= D_t u_*(t,x) h(x)= |t|U_*'(|t|) h(x)= \kappa(|t|)h(x),
\end{equation*}
where the function $\kappa$ can be defined by
\begin{equation}
  \label{eq:def-g_n}
\kappa \in C^\infty([0,\infty)),\qquad  \kappa(r) = r U_*'(r).
\end{equation}
We have
\begin{equation*}%\label{eq:g-properties}
\kappa'(0)=0, \qquad \kappa(1)=0 \qquad \text{and}\qquad \kappa'(1)=U_*''(1).
\end{equation*}
From \eqref{eq:approximation-u-n-h}, we  can  then  look for a solution to (\ref{eq:equivalence-F-ND}) of the form $\tilde u:=u_*-w_h +u$,  with $u $  and $h$ small. Then  $\tilde u = c_*=U_*(1)$ on $\de\O_*$ if and only if   $u=0$ on $\de\O_*$, by the definition of $u_*$ and since $\kappa(1)=0$. Moreover, since $D_t   u_* \equiv 0$ on $\de\O_*$, the condition $D_t  \tilde u \equiv 0$ on $\de\O_*$ enforces
$$
D_t  u(e_1,x)=D_t  w_h(e_1,x) = \kappa'(1) h(x)
$$
and therefore
\begin{equation}\label{eq:DeffM-ND--2}
h_u(x):=h(x)=\frac{D_t u (e_1,x)}{\kappa'(1)} , \qquad x \in \R .
\end{equation}
\end{Remark}\vspace{0.5ex}

From the Remark \ref{rem:const-sol},  we introduce the linear map
\begin{equation}
\label{def-M}
M: X_2^D \to X_2^{DN} \times  C^{2,\alpha}_{p}(\R),\qquad  M u = (M_1 u,h_u)
\end{equation}
with $h_u$ in  \eqref{eq:DeffM-ND--2}  and
\begin{equation*}%\label{eq:DeffM-ND--1}
[M_1 u](t,x)= u(t,x) - \kappa(|t|) h_u(x), \qquad (t,x) \in \O_*.
\end{equation*}
Clearly  $M_1 u \in X_2^{DN}$ for $u \in X_2^D$ and moreover, $h_u \in
C^{2,\alpha}_{p}(\R)$ for $u \in X_2^D$ by definition of $X_2^D$. Hence the linear map $M$ is well defined by (\ref{def-M}) and furthermore   $M: X_2^D \to X_2^{DN} \times  C^{2,\alpha}_{p}(\R)$ is a topological isomorphism, see  \cite[Lemma 3.4]{Fall-MinlendI-Weth4}.\\

We now define  the map
\begin{equation}\label{eq:DeffGl-ND}
G_\lambda: \cU  \to Y,\qquad G_\lambda = F_\lambda \circ M
\end{equation}
where
$$
\cU:= \left\{u \in X_2^D\::\: \, h_u(x)  >-1 \;\text{for $x \in \R$} \right \}.
$$
Then we have the equivalence
\begin{align}
  G_\lambda(u)= 0\quad &\Longleftrightarrow \quad F_\lambda(M_1 u, h_u)=0\nonumber\\
  &\Longleftrightarrow \quad \text{$u_*+M_1 u$ solves (\ref{eq:equivalence-F-ND}) with $h= h_u $.} \label{eq:Eqtosolve-ND}
  \end{align}

Obviously, $G_\lambda(0) = 0$ for all $\lambda>0$. We aim to find a branch of nontrivial solutions to the equation
$G_\lambda(u) = 0$ bifurcating from this trivial solution, which leads to the study of the linearization of $G_\lambda(u)$ around a point $u=0$.

\begin{Proposition} \label{sec:functional-setting-4}
 The map $G_\l :\cU \subset X_2^D \to Y $ defined by  \eqref{eq:DeffGl-ND} is of class $C^\infty$ and  for all $v\in X_2^D$,
\begin{align}\label{eequivla-ND}
DG_{\lambda}(0)v=\mathcal{L}_\lambda v:=\Delta_{t}v+\lambda\partial_{xx}v+\mu_{*}g'(u_*)v.
\end{align}
%Moreover, for  every $\lambda>0$, the operator $\mathcal{L}_\lambda= DG_\lambda : X_2^D \to Y$ is a Fredholm operator of index zero.
\end{Proposition}
\begin{proof}
That fact  that $G_\l$ is of class $C^\infty$ follows  similarly as in  \cite[Proposition 3.6]{Fall-MinlendI-Weth4}. To see (\ref{eequivla-ND}), we  first note that by  the chain rule,
\begin{equation*}%\label{eq:DGl0-compt}
DG_\l(0)v=\de_u F_\l(0,0)M_1 v+ \de_h F_\l(0,0)h_v \qquad \text{for $v \in X_2^D$.}
\end{equation*}
Furthermore,
$$
\partial_u F_{\lambda}(0,0)=\partial_u L_{\lambda}^{1}(u+u_*)\Bigl|_{u=0} =\Delta_{t}+\lambda\partial_{xx}+\mu_{*}g'(u_*)\textrm{id} . %=:\mathcal{L}_\lambda.
$$
By definition $M_1 v = v - w_{h_v}$ with $w_{h_v}(t,x) = \kappa(|t|)h_v(x)$, we get
\begin{equation}\label{eq:DGl0}
\de_u F_\l(0,0)M_1 v = \mathcal{L}_\lambda  M_1 v =\mathcal{L}_\lambda  v -  \mathcal{L}_\lambda  w_{h_v}.
\end{equation}
Next using  \eqref{eq:reldiffope-ND} % and \eqref{eq:ropeLtilde},
\[L^{1+ sh}_\l (u_*^{sh})=\tilde{L}^{1+ sh}_\l (u_*^{sh})+\mu_{*}g(u_*^{sh}).\]

Differentiating this with respect to $s$  and using \eqref{eq:approximation-u-n-h},  we get for fixed $h \in
C^{2,\alpha}_{p}(\R)$
\begin{align*}
  0&=\frac{d}{d_s} \Bigl|_{s=0} \left(L_{\lambda}^{1+sh}(u_*^{sh})\right)=\left(\frac{d}{d_s} \Bigl|_{s=0} \tilde{L}_{\lambda}^{1+sh}\right)(u_*)+  \tilde{L}_{\lambda}^{1}\left(\frac{d}{d_s} \Bigl|_{s=0} u_*^{sh}\right)-\mu_{*}g'(u_*)w_h,\\
&=  \left(\frac{d}{d_s} \Bigl|_{s=0} \tilde{L}_{\lambda}^{1+sh}\right)(u_*)- \tilde{L}_{\lambda}^{1}w_h -\mu_{*}g'(u_*)w_h,\\
  &=\left(\frac{d}{d_s}\Big|_{s=0} L_{\lambda}^{1+sh}\right)(u_*)-  \mathcal{L}_\lambda w_h,  %\label{zero-diff}
\end{align*}
with $w_h(t,x)= \kappa(|t|)h(x)$. We also have  $F(0, sh_v)= L_\lambda^{1+ sh_v}(u_*)$ and differentiate this with respect to $s$ implies
\begin{equation}\label{eq:DGl1}
\de_h F_\l(0,0)h_v = \Bigl(\frac{d}{ds}\Bigl|_{s=0} L^{1+ sh_v}_\l\Bigr)u_* = \mathcal{L}_\l w_{h_v}.
\end{equation}
Combining \eqref{eq:DGl0}  and \eqref{eq:DGl1}   gives $DG_\l(0)v= \mathcal{L}_\lambda  v $ for $v \in X_2^D$,   and  we obtain  \eqref{eequivla-ND}.
 \QED\end{proof}

In the next section, we will analyse the operator  $ \mathcal{L}_\lambda = DG_\lambda : X_2^D \to Y$  given by  \eqref{eequivla-ND}  and  provide the  required  assumptions for applying the Crandall-Rabinowitz Bifurcation theorem \cite{M.CR}. Before, we need some intermediate results.
\begin{Lemma}
  \label{regularity-lemma-1-ND}
Let $f \in C^{0,\alpha}_{p,rad}(\ov{\O_*})$ and $u \in C^{2,\alpha}_{p,rad}(\ov{\O_*})$ satisfy
  \begin{equation}
    \label{eq:regularity-x0-lemma-ND}
  \mathcal{L}_\lambda u = f \quad \text{in $\O_*$,}\qquad u = 0 \quad \text{on $\partial \O_*$.}
  \end{equation}
If $f \in Y$, then $u \in X_2^D$.
\end{Lemma}
\begin{proof}
By assumption (A)-(iii), we    have that the  mapping   $\mathcal{M}_*: Y \to Y$, $v \mapsto \mu_{*}g'(u_*)v$  is well defined and is a bounded linear operator. Let now  $f \in C^{0,\alpha}_{p,rad}(\ov{\O_*})$ and $u \in C^{2,\alpha}_{p,rad}(\ov{\O_*})$ such that  \eqref{eq:regularity-x0-lemma-ND} holds. Then
$$
  \Delta_t u+\lambda\partial_{xx} u = f-\mathcal{M}_*(u)\in Y \quad \text{in $\O_*$,}\qquad u = 0 \quad \text{on $\partial \O_*$.}
$$
Consequently, applying  \cite[Lemma 3.7]{Fall-MinlendI-Weth4} we deduce that $u \in X_2^D$.\QED
\end{proof}
\begin{Proposition}
\label{fredholm}
For every $\lambda>0$, the operator $\mathcal{L}_\lambda = DG_\lambda : X_2^D \to Y$ is a Fredholm operator of index zero.
\end{Proposition}
\begin{proof}
We observe that  $\mathcal{M}_*: Y \to Y$, $v \mapsto \mu_{*}g'(u_*)v$  defines  a bounded linear operator.   Since  furthermore the embedding $i: X_2^D \hookrightarrow Y$ is compact,   defining  $\widetilde{L}: X_2^D \to Y$, with  $\widetilde{L} v= \lambda v_{xx} + \D_t v$, we have that $$ \mathcal{L}_\lambda-\widetilde{L}=\mathcal{M}_*\circ i: X_2^D \rightarrow Y$$ is compact.  On the other hand, the operator $\widetilde{L}: X_2^D \to Y$, $\widetilde{L} v= \lambda v_{xx} + \D_t v$  is a topological isomorphism, (see  \cite[(3.29)]{Fall-MinlendI-Weth4})  and since the  Fredholm property and the Fredholm index are stable under compact perturbations, the proof is complete.
 \QED\end{proof}

\section{Study of the linearised operator $DG_\l(0)$ }\label{section 4}

In this section, we further analyse the operator  $ \mathcal{L}_\lambda = DG_\lambda : X_2^D \to Y$ in Proposition \ref{sec:functional-setting-4} and study  its spectral properties. \\

In the following we define for  $k \ge 0$, the spaces
$$
H^{k}_{0,rad}(B_1):= \{ u \in H^{k}_{0}(B_1)\::\: \text{$u$ is a radial function } \},
$$
 and we define $$ \mathcal{L}_{*,D}:  C^{2,\a}_{0, rad}(B_1)  \to C^{0,\a}_{rad}(B_1), \quad \quad \mathcal{L}_{*,D}(u):=-\Delta_{t}u-\mu_{*}g'(U_*)u$$ and denote  $V_{j,*} \in   C^{2,\a}_{0, rad}(B_1)$ the (radial)  eigenfunctions  of $\mathcal{L}_{*,D}$ with  the  corresponding  eigenvalues  $\gamma_{j, *}$ so that
\begin{equation} \label{z}
\begin{cases}
 V_{j,*}^{''}+\frac{N-1}{r}V_{j,*}^{'}+ \mu_{*}g'(U_*)V_{j,*}+\gamma_{j,*}V_{j,*}=0 &\mbox{in $(0, 1)$\,, }\\
V'_{j,*}(0)=0,   \quad  V_{j,*}(1)=0.
\end{cases}
\end{equation}
The quadratic form associated to  the operator $\mathcal{L}_{*,D}:  C^{2,\a}_{0, rad}(B_1)  \to C^{0,\a}_{rad}(B_1)$ is given by
$$ \mathcal{Q}_{*,D}:   H^{1}_{0,rad}(B_1) \to \R, \quad \quad   \mathcal{Q}_{*,D}(u)=\int_{B_1}\big(|\nabla  u|^{2}-\mu_{*}g'(U_*) u^{2}\big).$$
Moreover,
$$
\gamma_{1, *}=\inf\left\{  \frac{\mathcal{Q}_{*,D}(u)}{ \|u\|^2_{L^{2}(B_1)} }:   \quad  u\in H^{1}_{0,rad}(B_1)  \right\}.
$$
%Here $\gamma_{1, m}$ is just $\gamma_{m}$ given in Remark \ref{eq:RKmain-ND}.

Since $\mathcal{Q}_{*,D}$ is considered among radially symmetric functions, we can write
 \begin{align}
    \label{gamma12}
\gamma_{1, *} &= \inf_{u\in \cH(0,1)} \frac{ \int_{0}^{1} r^{N-1} [u'(r)^2 - \mu_{*}g'(U_*)u^2(r)]dr }{   \int_{0}^{1}  r^{N-1} u^2(r)dr  },
\end{align}
with $\cH(0,1):= \{u \in H^1(0,1)\::\: \text{$u(1)=0$ in trace sense}\}$ and its norm is given by:

\begin{equation*} \label{radialnorm} \| u \|_{\cH(0,1)} = \left ( \int_{0}^{1} r^{n-1} \left[(\partial_{r} u)^{2}+ u^{2}\right] dr \right )^{1/2}. \end{equation*}

In the sequel, we will often use $\mathcal{Q}_D$ for $\mathcal{Q}_{*,D}$ if there is no confusion. With this, we have

\begin{Lemma} \label{legamma1}
There holds: $\gamma_{1, *} < 0$.
\end{Lemma}
\begin{proof}
It suffices to find $\psi \in  H^{1}_{0,rad}(B_1)$ such that $\mathcal{Q}_D<0$. Since $\mathcal{Q}_D$ is considered among radially symmetric functions, we can write the quadratic form as
\begin{align*}
    \mathcal{Q}_D(\psi)&=\int_{B_1}\big[|\nabla\psi|^{2}-\mu_{*}g'(U_*)\psi^{2}\big]\\
    &=\omega_{N}\int_0^1r^{N-1}\big[\psi'(r)^{2}-\mu_{*}g'(U_*)\psi(r)^{2}\big]dr.
\end{align*}
Recall that the function $U_*$ solves  \eqref{eq:ODE-eigenvalue}. Hence  differentiating the equation
       \[- U_*''-\frac{N-1}{r}U_*'-\mu_{*}g(U_*)=0, \]
we obtain
  \begin{equation}\label{eq368}
  -U_*'''(r)-\frac{N-1}{r}U_*''(r)+\frac{N-1}{r^{2}}U_*'(r)-\mu_{*}g'(U_*)U_*'(r)=0.
\end{equation}
We also have
\begin{align*}
\int_0^1r^{N-1}U_*'''(r)U_*'(r)dr&=\int_0^1r^{N-1}U_*'(r)dU_*''(r)=-\int_0^1U_*''(r)d\big(r^{N-1}U_*'(r)\big)\\
&=-\int_0^1r^{N-1}U_*''(r)^{2}dr-(N-1)\int_0^1r^{N-2}U_*'(r)U_*''(r)dr.
\end{align*}
Therefore  multiplying  (\ref{eq368}) by $r^{N-1}U_*'(r)$ and integrating, we obtain
\begin{equation*}
\int_0^1r^{N-1} \left[ U_*''(r)^{2}-\mu_{*}g'(U_*)U_*'(r)^{2}\right]dr=-(N-1)\int_0^1r^{N-3}U_*'(r)^{2}dr.
\end{equation*}
We can take the test function $U_*'(r)\in  H^{1}_{0,rad}(B_1) $ obtaining:
  \begin{align*}
  \mathcal{Q}_D(U_*'(r)) &=\omega_{N}\int_0^1r^{N-1}\big[U_*''(r)^{2}-\mu_{*}g'(U_*)U_*'(r)^{2}\big]dr\\
     &=-(N-1)\omega_{N}\int_0^1r^{N-3}U_*'(r)^{2}dr.
\end{align*}
Recalling $N\geq 2$, the proof is complete.
\QED
\end{proof}

From Lemma \ref{legamma1},  there exists a number $\ell_{0} \in \N$ such that
\begin{equation*}%\label{eqseqeigen}
\gamma_{1,*}<\gamma_{2, *}<\cdots<\gamma_{ \ell_{0}, *}<0<\gamma_{ \ell_{0}+1, *}<\cdots.
\end{equation*}

In the  following, we denote by $V_*:=V_{1,*}$, the eigenfunction corresponding to the first eigenvalue  $\gamma_{1, *}$ defined  by \eqref{gamma12} and consider the function
\begin{equation*}
v_*(t, x):=V_*(t)\cos(x).
\end{equation*}
We also set
\begin{align*}%\label{lamdam}
\lambda_* :=-\gamma_{1,*}.
\end{align*}
\begin{Proposition}\label{propCR-ND-Dir}
We have the following properties.
\begin{itemize}
\item[(i)] The kernel $N(\cL _*)$ of the operator  $\cL _*: =D G_{\lambda_*}(0)$ in Proposition  \ref{sec:functional-setting-4}  is spanned  by $v_*(t, x)= V_*(|t|)\cos(x)$, where  $V_*$ is  minimizer of  \eqref{gamma12}.
%where
%$$\sqrt{1-\alpha_\ell(m)}(2m+1)=\frac{1}{2}+\ell.$$
\item[(ii)] The range of $\cL_* $ is given by
$$
R(\cL_* )=   \left \lbrace  w\in Y: \int_{\O_* } v_{*}(t,x)w(t,x)\,dxdt=0\right\rbrace.
$$
\item[(iii)]  Moreover,
\begin{equation*}%\label{eq:transversality-c8ond4-ND-Dir}
\partial_\lambda \Bigl|_{\lambda=\l_* }\cL _*(v_*)\not  \in \; R(\cL _*).\\
\end{equation*}
\end{itemize}
\end{Proposition}

\begin{proof}
(i) Let $v\in X_2^D $  such that $ \mathcal{L}_{*}v=0$ in $ \O_*$.  We  expand $v$ as a uniformly convergent Fourier series in the $x$-variable of the form $v(t,x)=\sum \limits_{\ell=0}^\infty v_\ell(|t|)\cos(\ell x)$. Then for every $\ell \ge 0$, the coefficient function $v_\ell(t):= \frac{1}{\sqrt{2\pi}}\int_{0}^{2\pi} v(t,x)\cos (\ell x)\,dx$ is an eigenfunction of the eigenvalue problem
\begin{equation*} %\label{eq:ell-eigenvalue-problem}
\quad
\left\{\begin{aligned}
&v^{''}+\frac{N-1}{r}v^{'}+ \mu_*g'(U_*)v  = \lambda_* \ell ^2 v\quad \text{in $(0,1)$,}\\
&v'(0)=v(1)=0.
\end{aligned}
\right.
\end{equation*}
Therefore we obtain from  \eqref{z},
\begin{equation} \label{bifpara1}
\lambda_* \ell^2 =-\gamma_{j,*},  \quad \textrm{for some $j \in \{1, \cdots, \ell_{0}\}$}.
\end{equation}
Obviously  \eqref{bifpara1} does not hold for $\ell=0$.  In addition if \eqref{bifpara1} is valid for some $j> 1$, then $ \ell^2 =\frac{\gamma_{j,*}}{\gamma_{1,*}}<1$   after  recalling that $\gamma_{1,*}$ and  $\gamma_{j,*}$ are both  negative. Hence $\ell =0$ which is impossible.  The equality \eqref{bifpara1} therefore only holds for $j=1$  and we obtain  $\ell =1$.

To prove (ii), we let   $ w \in R(\mathcal{L}_{*})  \subseteq  Y$. Then there exists $u \in  X^D_2$ such that
$$
\mathcal{L}_{*}u= w,
$$
that is,
\begin{equation}\label{eq:diffeq33}
\left\{\begin{aligned}
\D_{\t}u + \lambda_* \partial_{xx}u +  \mu_{*}g'(U_*) u&=w \quad\text{in}\quad \O_*,\\
u &= 0\quad\text{on}\quad \partial\O_*.
 \end{aligned}
\right.
\end{equation}
We define $\cB: H_0^1(\O_*) \times H^1_0(\O_*)\to \R,$
$$
\cB(u,v) =\int_{\O_*} [\n_t u \cdot \n_t v+\l_*\de_x   u \de_x v]-  \mu_{*}\int_{\O_*}g'(U_*) u v.
$$
Multiply  \eqref{eq:diffeq33} by $\vp\in C^1(\ov\O_*) $ and integrate by parts to have
$$
\cB(u,\vp)=\int_{\O_*}w\vp.
$$
It is clear that  $\cB(u,v_*)=0$ and  we  deduce   $\int_{\O_*}w v_* =0 $, so that
\be\label{eq.range}
R(\mathcal{L}_{*}) \subseteq  \left \lbrace  w\in Y: \int_{\O_* } v_{*}(t,x)w(t,x)\,dxdt=0\right\rbrace.
\ee
Finally, we note that  $R(\mathcal{L}_{*})$ has codimension one by (i) and since $\mathcal{L}_{*}$ is Fredholm of index zero by Proposition~\ref{fredholm}. This together  with  \eqref{eq.range}  gives (ii).

Finally, it is clear from \eqref{eequivla-ND} that
$$
\de_\l\big|_{\l=\l_{*}} DG_{\l }(0) v_* = \de_{xx}v_*=-v_*.
$$
The  proof is complete.
 \QED
\end{proof}

\section{Proof of Theorem \ref{Theo1-ND}  }\label{eq:ProofTheo1-ND}
The proof of Theorem  \ref{Theo1-ND} is achieved  by applying the Crandall-Rabinowitz Bifurcation theorem to solve the equation
\begin{align*}%\label{eq:maptsolvG}
 G_\lambda(u)=0,
\end{align*}
where   $G_{\l}: \cU   \rightarrow Y$  is   defined by
\eqref{eq:DeffGl-ND}. \\

We consider the smooth  map  $G_{\l_*}: \cU  \rightarrow   Y$ and  set
\begin{align*}%\label{eq:spaorthokernel}
\mathcal{X}_{*} := \left\lbrace  v \in  X^D_{2}: \int_{\O_* } v(t,x)v_{*}(t,x)\,dxdt =0\right\rbrace.
\end{align*}
By Proposition \ref{propCR-ND-Dir}  and  the Crandall-Rabinowitz Theorem (see \cite[Theorem 1.7]{M.CR}), we then find ${\e_*}>0$ and a smooth curve
$$
(-{\e_* },{\e_* }) \to  (0,\infty) \times \mathcal{U} \subset  \R_+ \times  X^D_{2}, \qquad s \mapsto (\lambda_*(s), \varphi^*_s)
$$
such that
\begin{enumerate}
\item[(i)] $G_{\lambda_*(s)} (\varphi^*_s)=0$ for $s \in (-{\e_* },{\e_* })$,
\item[(ii)] $ \lambda_* (0)= \l_{*}$, and
\item[(iii)]  $\varphi^*_s = s v_{*}+ s \o_*(s) $ for $s \in (-{\e_* },{\e_* })$ with a smooth curve
$$
(-{\e_* },{\e_* }) \to \mathcal{X}_* ^{\perp}, \qquad s \mapsto \o_*(s)
$$
satisfying $ \o_*(0) =0$
and
$$\int_{\O_* } \o_*(s) (t,x) v_{*}(t,x)\,dxdt=0.$$
\end{enumerate}

\subsection*{Proof of Theorem \ref{Theo1-ND} (completed)}
\label{sec:proof-theor-refth}

Recalling  \eqref{eq:Eqtosolve-ND}, we see that, since  $G_{\lambda_*(s)} (\vp^*_{s})=0$ for every  $s \in (-{\e_*},{\e}_*)$, the function
\begin{align}
 \ti u_s ( \t,x)&:= u_*(\t,x) + [M_1   \varphi^*_{s}]({\t},x) \nonumber\\
  &= u_*(\t,x) + \varphi^*_{s}({\t},x)-\kappa(|t|)h_{\vp^*_{s}}(x)  \label{eq:Solutionfinal}
\end{align}
solves  (\ref{eq:equivalence-F-ND}) with
\begin{equation}
  \label{eq:def-h-v-phi-s}
h_{\vp^*_{s}}(x) = \frac{1}{U''_*(1)}{D}_t \vp^*_{s}(e_1,x)
\end{equation}
and $\kappa(r)=r U_*'(r)$ as defined in (\ref{eq:DeffM-ND--2}) and (\ref{eq:def-g_n}). Hence  by   \eqref{eqreladiffopets-ND-2} and  \eqref{eq:changsoluD}, the function
$$
(t,x )\mapsto w_s(t,x)=\ti u_s(h^*_{s}(x)t,x)
$$
solves  \eqref{h-Neu-over} with $\mu=\frac{\mu_*}{\l_*(s)} $ and
\begin{equation}\label{eq:expphhhm}
 h^*_{s}(x)=\frac{1+h_{\varphi^*_{s}}(x)}{\sqrt{\l_*(s)}} \qquad \text{for $x \in \R$.}
\end{equation}
Moreover, by (iii),
\begin{equation}
  \label{eq:varphi-s-exp}
\varphi_{s}(t,x)= s v_*(t,x)+ o(s)= s V_*(|t|)\cos( x)+ o(s),
\end{equation}
where $o(s) \to 0$ in $C^2$-sense in $\ov{\O_*}$ as $s \to 0$. Hence using \eqref{eq:def-dt-u},
\begin{equation*}
  \begin{split}
    D_t \varphi_{s}(e_1,x)&= s V'_*(1)\cos( x)+ o(s),
  \end{split}
\end{equation*}
where $o(s) \to 0$ in $C^1$-sense in $\ov{\O_*}$ as $s \to 0$. Using this, we find  with  \eqref{eq:def-h-v-phi-s} that
\begin{equation}
\begin{split}
  h_{\varphi^*_{s}} (x)&= s \frac{V'_*(1) }{ U''_*(1)}\cos( x)+ o(s) \qquad \text{as $s \to 0$}
\end{split}
  \label{eq:varphi-s-der-exp}
   \end{equation}
and therefore,  using  \eqref{eq:expphhhm} it follows that
\begin{align*}
  h^*_{s}(x)&=\frac{1+h^*_{\varphi_{s}}}{\sqrt{\l_*(s)}}= \frac{1}{\sqrt{\l_*(s)}} + s  \frac{V'_*(1) }{ U''_*(1) \sqrt{\l_*(s)}}\cos( x)+ o(s)
%&=   \frac{1}{\sqrt{\l_m(s)}} + s  \frac{V'_m(1) }{ U''_m(1) \sqrt{\lambda_{m}}}\cos( x)+ o(s)
\end{align*}
as $s \to 0$.

Finally, by (\ref{eq:Solutionfinal}), (\ref{eq:varphi-s-exp}) and (\ref{eq:varphi-s-der-exp}),
\begin{align}
&w_s\left(\frac{t}{h^*(s)},x\right)= \ti u_s ( \t,x) = u_*(\t,x) + \varphi^*_{s}({\t},x) - \kappa(|t|)h_{\vp^*_{s}}(x) \nonumber\\
  &=U_*(|t|) + s \Bigl( V_{*}(|t|)- \frac{V'_*(1)}{U''_*(1)}|t|U_*'(|t|) \Bigr)\cos (x) + o(s),\nonumber
\end{align}
where $o(s) \to 0$ in $C^1$-sense on $\O_*$.
Then one deduces the desired constant in  Theorem~\ref{Theo1-ND}, which completes  the  proof.

\QED

%\bibliographystyle{plain} %Options: amsplain, elsarticle-harv
%\bibliography{references}
\end{document}